\documentclass[11pt,fleqn]{amsart}
\usepackage[english]{babel}
\usepackage[top=1.5in, bottom=1.5in, left=1.2in, right=1.1in]{geometry}
\usepackage{amsmath,amsthm,amsfonts,amssymb,soul}
\usepackage[title]{appendix}
\usepackage{bm}
\usepackage{cite}
\usepackage{color}
\usepackage{lipsum}
\usepackage{enumitem}
\usepackage{epsfig}
\usepackage{mathrsfs}
\usepackage{scalerel}[2014/03/10]
\usepackage[usestackEOL]{stackengine}
\usepackage[svg]{xcolor}
\usepackage{varioref,graphicx,epstopdf,subfigure}
\usepackage{verbatim}
\usepackage{csquotes}
\usepackage{hyperref}

\newtheorem{theorem}{Theorem}[section]

\newtheorem{lemma}[theorem]{Lemma}

\theoremstyle{definition}

\newtheorem{remark}[theorem]{Remark}

\renewcommand{\d}{{\rm d}}
\renewcommand{\u}{u_{\rm rad}^\delta}
\newcommand{\uiso}{u_{\rm rad}^0}
\newcommand{\C}{\mathbb{C}}
\newcommand{\E}{\mathfrak{E}}
\renewcommand{\H}{\mathcal{H}}
\renewcommand{\L}{\mathfrak{L}_\delta}

\newcommand{\Q}{\mathfrak{Q}_{\rm rad}^\delta}
\newcommand{\Qiso}{\mathfrak{Q}_{\rm rad}^0}
\newcommand{\R}{\mathbb{R}}

\renewcommand\Re[1]{\operatorname{Re}\left\{#1\right\}}
\renewcommand\Im[1]{\operatorname{Im}\left\{#1\right\}}

\newcommand{\e}{\varepsilon}

\def\avgint{\,\ThisStyle{\ensurestackMath{%
			\stackinset{c}{.2\LMpt}{c}{.5\LMpt}{\SavedStyle-}{\SavedStyle\phantom{\int}}}%
		\setbox0=\hbox{$\SavedStyle\int\,$}\kern-\wd0}\int}

\newcommand{\intP}[1]{\int_{\R^2}#1\,\d x}
\newcommand{\peta}{\partial_{\eta}}
\newcommand{\peeta}{\partial_{\eta\eta}}

\newcommand{\abs}[1]{{\left\vert #1\right\vert}}

\fontfamily{lmtt}\selectfont

\begin{document}
\title{On the stability of radial solutions to an anisotropic Ginzburg-Landau equation}
\date{\today}

\author{Xavier Lamy}
\address{Institut de Math\'ematiques de Toulouse, UMR 5219, Universit\'e de Toulouse, CNRS, UPS IMT, F-31062 Toulouse Cedex 9, France.}
\email{xavier.lamy@math.univ-toulouse.fr}

\author{Andres Zuniga}
\address{Instituto de Ciencias de la Ingenier\'ia (ICI), Universidad de O'Higgins (UOH), Rancagua, Chile.
} 
\email{andres.zuniga@uoh.cl}

\begin{abstract}
We study the linear stability of entire radial solutions $u(re^{i\theta})=f(r)e^{i\theta}$, with positive 
increasing 
profile $f(r)$, to the anisotropic Ginzburg-Landau equation
\begin{align*}
-\Delta u -\delta (\partial_x+i\partial_y)^2\bar u =(1-|u|^2)u,\quad -1<\delta <1,
\end{align*}
which arises in various liquid crystal models. In the isotropic case $\delta=0$, Mironescu showed that such solution is nondegenerately stable. We prove stability of this radial solution in the range $\delta\in (\delta_1,0]$ for some $-1<\delta_1<0$, and instability outside this range. In strong contrast with the isotropic case, stability with respect to higher Fourier modes is \emph{not}  a direct consequence of stability with respect to lower Fourier modes. In particular, in the case where $\delta\approx -1$, lower modes are stable and yet higher modes are unstable.
\end{abstract}	

\maketitle

\section{Introduction}

Given $\delta\in (-1,1)$ and $u:\R^2\to\C$, we consider the anisotropic energy
\begin{align}\label{eq:E}
	\E[u]=\intP{\frac{1}{2}|\nabla u|^2+\frac{\delta}2\Re{(\peta\bar u)^2}+\frac{1}{4}(1-|u|^2)^2},\qquad \text{ where }\;\partial_\eta=\partial_x+i\partial_y.
\end{align}
Minimizers and stable critical points of $\E$ are relevant in describing 2D point defects (or 3D straight-line defects) in some liquid crystal configurations (e.g. smectic-$C^*$ thin films \cite{phillips13} and nematics close to the Fr\'eedericksz transition \cite{clerc14}). This energy can also be viewed as a toy model to understand intricate phenomena triggered by elastic anisotropy in the more complex Landau-de Gennes energy  \cite{kitavtsev16}.

\begin{remark}\label{r:ksb}
The anisotropic term $\Re{(\peta\bar u)^2}$ can be rewritten as
\begin{align*}
\Re{(\peta\bar u)^2}=(\nabla\cdot u)^2 - (\nabla\times u)^2,
\end{align*}
so that, in view of the identity $\abs{\nabla u}^2=(\nabla\cdot u)^2+(\nabla\times u)^2 -2\det(\nabla u)$, energy~\eqref{eq:E} differs from 
\begin{align*}
\widetilde\E[u]
=\int \frac{k_s}{2}(\nabla\cdot u)^2 + \frac{k_b}{2}(\nabla\times u)^2 +\frac{1}{4}(1-\abs{u}^2)^2,\qquad k_s=1+\delta,\; k_b=1-\delta,
\end{align*}
only by the integral of the null Lagrangian $\det(\nabla u)$. This is precisely the form that appears in \cite{phillips13} where minimizers of 
\begin{align}\label{eq:Etildeeps}
	\widetilde\E_\e[u]
	=\int_{\Omega} \frac{k_s}{2}(\nabla\cdot u)^2 + \frac{k_b}{2}(\nabla\times u)^2 +\frac{1}{4\e^2}(1-\abs{u}^2)^2
\end{align}
are investigated in the limit as $\e\to 0^+$ in a bounded planar domain $\Omega$.
\end{remark}
Critical points of $\E$ are solutions of the Euler-Lagrange equation 
\begin{equation}\label{eq:GL}
\begin{aligned}
	 \L u &= (|u|^2-1)u\qquad\text{in }\R^2\\
	\L u &:=\Delta u+\delta\,\peeta\bar u.
\end{aligned}
\end{equation}
We are interested in symmetric solutions of the form
\begin{align}\label{eq:symansatz}
	&u(re^{i\theta})=f(r)e^{i\alpha}e^{i\theta} \qquad\text{for some }\alpha\in\R,
\end{align}
with a radial profile $f(r)$ satisfying
\begin{align}\label{eq:condf}
	&f(0)=0,\quad \lim_{r\to+\infty} f(r)=1,\quad \abs{f(r)}>0\quad\forall r\in (0,\infty).
\end{align}
Formally, one can always look for solutions of \eqref{eq:GL} in the form \eqref{eq:symansatz} (as a consequence of the $O(2)$-invariance of $\E$), and $f$ must solve
\begin{align*}
	Tf +\delta e^{-2i\alpha}T\bar f =\left(\abs{f}^2-1\right) f,\qquad T=\frac{d^2}{dr^2} +\frac 1r \frac{d}{dr} -\frac 1{r^2}.
\end{align*}
At this point we see a fundamental difference with respect to the isotropic case $\delta=0$. If $\delta=0$, one can find solutions as above for a real-valued function $f$, which moreover does not depend on $\alpha$. 
In the anisotropic case $\delta\neq 0$, as remarked in  \cite{clerc14}, the function $f$ can be  real-valued only if 
 $\alpha\equiv 0$ modulo $\pi/2$. In that case, the existence and uniqueness of a solution satisfying \eqref{eq:condf} follows from the case $\delta=0$ (see~\cite{herve94,chen94}).
Otherwise, the function $f$ must be complex valued. 

\begin{remark}
Another difference with respect to the isotropic case is that for $\delta\neq 0$ the Ansatz $u(re^{i\theta})=f(r)e^{im\theta}$ cannot provide a solution when the winding number $m$ is $\neq 1$. 
\end{remark}

In \cite{clerc14}, the core energies of the two symmetric solutions corresponding to $\alpha =0,\pi/ 2$ are compared, to find that the lowest energy corresponds to $\alpha=0$ for $\delta<0$ and  $\alpha=\pi/2$ for $\delta>0$. In view of Remark~\ref{r:ksb} this is consistent with the fact that $\nabla\times e^{i\theta}=0$, while $\nabla\cdot ie^{i\theta}=0$; indeed, for $\delta<0$ the energy $\widetilde\E[u]$ in Remark~\ref{r:ksb} penalizes more strongly the term $(\nabla\times u)^2$ than the term $(\nabla\cdot u)^2$, since in this case $k_b=1-\delta >k_s=1+\delta$. In \cite[Proposition~3.1]{phillips13} the authors use this to show that minimizers of \eqref{eq:Etildeeps} behave like $e^{i\alpha}e^{i\theta}$ around point defects, with $\alpha\equiv 0$ (resp. $\pi/2$) modulo $\pi$ if $\delta<0$ (resp. $\delta>0$).  
These results tell us, for $\delta\neq 0$, which one is the minimizing behavior at infinity. 

Here, in contrast, we fix the far-field behavior and investigate the  local  stability of radial solutions with respect to compactly supported perturbations.
For the isotropic case $\delta=0$, this study has been performed in \cite{mironescu95stab} (see also \cite{delpino04}), and the radial solution is stable. In the anisotropic situation $\delta\neq 0$ we find that the corresponding symmetric solution stays stable for negative $\delta$ close to zero and it loses stability for $\delta$ either positive or close to minus one (see Theorem~\ref{t:main} for precise statements). 

It can be readily seen that the case $\alpha=\pi/2$ corresponds to $\alpha=0$, after changing the sign of $\delta$. Accordingly, we only treat the case where $\alpha=0$. That is,
we investigate the linear stability of solutions $u$ of the form
\begin{align}\label{eq:urad}
	\u(r,\theta)=f(r)e^{i\theta},\quad f\colon (0,+\infty)\to (0,+\infty)\quad\text{ with }\quad f(0)=0,\;\;\lim_{r\to+\infty}f(r)=1.
\end{align}
Let us note that the equation satisfied by $\u$,~\eqref{eq:GL}, reduces to the following ODE for $f$
\begin{align}\label{eq:falpha0}
	(1+\delta) Tf =(f^2-1)f,\qquad T=\frac{d^2}{dr^2} +\frac 1r \frac{d}{dr} -\frac 1{r^2}.
\end{align}
As pointed out in \cite{clerc14}, the rescaling of the variable by $(1+\delta)^{\frac 12}$ simplifies~\eqref{eq:falpha0} to the standard ODE corresponding to the isotropic case $\delta=0$. Whence, existence and uniqueness of $f$ follow from~\cite{herve94,chen94}. Moreover, it is known that $f$ takes values in $(0,1)$ and is strictly increasing.

The second variation of the energy $\E$ around $\u$ is the quadratic form
\begin{align}
	\Q[v]
	&=\intP{|\nabla v|^2+\delta\Re{(\peta\bar v)^2}-(1-|\u|^2)|v|^2+2\,(\u\cdot v)^2}\notag\\
	&=\intP{|\nabla v|^2+\delta\Re{(\peta\bar v)^2}-(1-f^2)|v|^2+2f^2\,(e^{i\theta}\cdot v)^2}\label{eq:Qrad}
\end{align}
associated to the linear operator obtained by linearizing~\eqref{eq:GL} around $\u$:
\[
	\mathcal L(\u)[v]=-\L v-(1-|\u|^2)v+2\,(\u\cdot v)\,\u,
\]
where $u\cdot v:=\Re{u\bar v}$ denotes the standard inner product of complex-valued functions.

Taking into account the asymptotic expansion $f(r)=1+O(r^{-2})$ as $r\to\infty$ (see~\cite{herve94,chen94}), it follows that the energy space of $\Q$ naturally corresponds to
\begin{align*}
	\H:=\left\lbrace v\in H^1_{loc}(\mathbb R^2)\colon \intP{\abs{\nabla v}^2+\frac{1}{r^2}\abs{v}^2+ (e^{i\theta}\cdot v)^2} <+\infty \right\rbrace.
\end{align*}
Also, the translational invariance of $\E$ readily provides two elements of $\H$ at which $\Q$ vanishes, namely
\begin{align*}
\partial_x \u &= e^{i\theta}\left( f' \cos\theta - i\frac fr \sin\theta \right),&
 \partial_y \u &=e^{i\theta} \left(f'\sin\theta +i\frac fr\cos\theta\right),
\end{align*}
and the linear space they generate is denoted by
\begin{align*}
K_0=\mathrm{span}\lbrace \partial_x \u,\partial_y \u\rbrace.
\end{align*}
Our main result shows that the symmetric solution $\u$ is stable when $\delta\leq 0$ is small, and unstable otherwise:
\begin{theorem}\label{t:main}
Let $\u$ denote the radial solution~\eqref{eq:urad} of the anisotopric Ginzburg-Landau equation~\eqref{eq:GL}, and let $\Q$ denote the quadratic form~\eqref{eq:Qrad} associated to the energy $\E$ around $\u$. Then, there exists a unique number $\delta_1\in(-1,0)$ such that
\begin{itemize}
\item for every $\delta\in (\delta_1,0]$, $\u$ is nondegenerately stable: namely,
\begin{align*}
	\Q[v]>0\qquad\text{for all } v\in H\setminus K_0,
\end{align*}
\item for every $\delta\in (-1,\delta_1)\cup (0,1)$, $\u$ is linearly unstable: namely,
\begin{align*}
\Q[v]<0\qquad\text{for some }v\in H.
\end{align*}
\end{itemize}
\end{theorem}

\begin{remark}
The most relevant range from the stand point of physics is $\delta\in (-1,0]$ since for $\delta>0$ the far-field behavior corresponding to $\alpha=0$ is non-minimizing,  and this translates here into instability of the radial solution.
\end{remark}

\begin{remark}
In the stability range $\delta\in (\delta_1,0]$, a contradiction argument as in \cite[Lemma~3.1]{delpino04} provides a coercivity estimate of the form 
\begin{align*}
	\Q[v]\geq C(\delta) \intP{|\nabla v|^2}\qquad
	\forall v\in K_0^\perp \colon \int_{\mathbb S^1} (ie^{i\theta})\cdot v(re^{i\theta})\, d\theta=0\;\;\forall r>0,
\end{align*}
where $\perp$ denotes orthogonality in $\H$.
Using this coercivity for $\delta=0$, one can deduce stability for small negative $\delta$ via a relatively simple perturbation argument,  combined with properties of the lower modes in \S~\ref{s:low}. Instead, we will give a more quantitative proof,  which provides an explicit range for stability: we deduce that $\delta_1 \leq -1/\sqrt 5$.
\end{remark}

Our proof of Theorem~\ref{t:main} follows the general strategy of \cite{mironescu95stab}: we decompose $v$ into Fourier modes
\begin{align*}
v =e^{i\theta}\sum_{n\in\mathbb Z}w_n(r)e^{in\theta}.
\end{align*}
 and we are led to studying the sign of $\Q$, separately, for each mode
\begin{align*}
e^{i\theta}\left(w_n(r)e^{in\theta} + w_{-n}(r)e^{-in\theta}\right).
\end{align*}
As in~\cite{mironescu95stab}, the lower modes $n=0$ and $n=1$ play a special role. They can be studied via an appropriate decomposition already used in \cite{mironescu95stab} (see also \cite{delpino04}). For any $\delta\in(-1,0]$ we find that these lower modes are stable, while for $\delta >0$ the mode corresponding to $n=0$ is unstable. 

A major difference of the present work compared to~\cite{mironescu95stab} (or similar results in \cite{ignat15stabhedgehog,ignat15instab2d,ignat16stab2d}) pertains to the higher modes $n\geq 2$. In contrast with the cited works, stability for the higher modes is not an obvious consequence of stability for the lower modes. More precisely in the isotropic case we have 
\begin{align*}
\Qiso\left[ e^{i\theta}\left(w_+(r)e^{in\theta} + w_{-}(r)e^{-in\theta}\right)\right] \geq 
\Qiso \left[ e^{i\theta}\left(w_+(r)e^{i\theta} + w_{-}(r)e^{-i\theta}\right)\right]\quad\forall n\geq 1,
\end{align*}
but for $\delta\neq 0$ this is not valid anymore, see \eqref{eq:QnQ1}. This feature is new and specific to the anisotropic case $\delta\neq 0$. Our strategy to study the sign of these higher modes is based on the same decomposition used for $n=1$, and a careful balance of the contributions of additional terms, which end up causing instability for $\delta\approx -1$. 
\medskip

The article is organized as follows. In Section~\ref{s:split} we recall the splitting property of the quadratic form $\Q$ with respect to Fourier expansion. In Section~\ref{s:low} we study the stability of lower modes, and in Section~\ref{s:high} the instability of higher modes.
In Section~\ref{s:proof} we give the proof of Theorem~\ref{t:main}. In addition, we included Appendix~\ref{a:stabA} to recall the details of the decomposition used to study the lower modes, adapted to our notations.

\medskip

\textbf{Acknowledgements}
\medskip

XL is partially supported by ANR project ANR-18-CE40-0023 and COOPINTER project IEA-297303. AZ is supported by ANID Chile under the grant FONDECYT de Iniciaci\'on en Investigaci\'on $N^{\circ}$ 11201259.

\section{Fourier splitting}\label{s:split}

Recall that $f (r) = f_0 ((1 + \delta)^{-\frac 12} r)$ where $f_0$ is the classical Ginzburg-Landau vortex profile
corresponding to the case $\delta = 0$. That is, the unique solution of
\begin{align}\label{eq:f0}
	f_0'' +\frac 1r f_0' -\frac{1}{r^2}f_0 =-(1-f_0^2)f_0,\quad f_0>0\text{ on }(0,+\infty),\quad f_0(0)=0,\;\; \lim_{r\to +\infty}f_0(r)=1.
\end{align}
We rescale variables and consider $\mathcal Q^\delta[v] = \Q[\tilde v]$ where
$\tilde v(\tilde x) = v((1 +\delta)^{-\frac 12}\tilde x)$, so that
\begin{align}\label{eq:Q}
\mathcal Q^\delta[v] 
&=\intP{|\nabla v|^2+\delta\Re{(\peta\bar v)^2}
+(1+\delta)\left\lbrace 2f_0^2\,(e^{i\theta}\cdot v)^2-(1-f_0^2)|v|^2\right\rbrace},
\end{align}
which corresponds to the second variation of the appropriately rescaled energy around $\uiso$.
Following \cite{mironescu95stab} we decompose $v$ using Fourier series, as 
\begin{align}\label{eq:Fourier}
v=e^{i\theta} w =e^{i\theta}\sum_{n\in\mathbb Z}w_n(r)e^{in\theta},
\end{align}
where we have conveniently shifted the index $n-1\mapsto n$.

This decomposition provides a \enquote{diagonalization} of the linearized operator:

\begin{lemma}\label{l:decomp}
The quadratic form~\eqref{eq:Q} splits as
\begin{align*}
\mathcal Q^\delta[v] = \mathcal Q^\delta\left[w_0(r)e^{i\theta}\right] + \sum_{n\geq 1} \mathcal Q^\delta\left[e^{i\theta}\left(w_n(r)e^{in\theta} +w_{-n}(r)e^{-in\theta}\right)\right].
\end{align*}
\end{lemma}

\begin{proof}[Proof of Lemma~\ref{l:decomp}.]
Lemma~\ref{l:decomp} essentially asserts that the family of functions
\begin{align}\label{eq:fouriermodes}
w_0(r)e^{i\theta},\quad \{e^{i\theta}\left(w_n(r) e^{in\theta} + w_{-n}(r) e^{-in\theta}\right):\;n\geq 1\},
\end{align}
is orthogonal for the quadratic form $\mathcal Q$. This quadratic form \eqref{eq:Q} is composed of three terms. For the first term, 
\begin{align*}
	\intP{|\nabla v|^2},
\end{align*}
the orthogonality of \eqref{eq:fouriermodes} is a standard fact (recall e.g. in \cite{mironescu95stab}). For the third term,
\begin{align*}
	\intP{\left\lbrace f_0^2\,(e^{i\theta}\cdot v)^2-(1-f_0^2)|v|^2\right\rbrace},
\end{align*}
the orthogonality of \eqref{eq:fouriermodes} is proved in \cite{mironescu95stab}. The novelty here, with respect to \cite{mironescu95stab}, concerns the anisotropic term
\begin{align*}
\intP{\Re{(\peta\bar v)^2}}.
\end{align*}
The orthogonality of \eqref{eq:fouriermodes} for this anisotropic term, as a matter of fact, follows from the calculations in \cite[\S~3.2]{phillips17}. As our notations are different, we sketch a proof here for the reader's convenience.

We compute
\begin{align*}
\partial_\eta\bar v &=e^{i\theta}\partial_r \bar v +\frac{ie^{i\theta}}{r}\partial_\theta \bar v 
=\sum_{n\in\mathbb Z} \left(\bar w_n' +\frac{1+n}{r}\bar w_n\right) e^{-in\theta},
\end{align*}
and deduce, using the orthogonality of $\lbrace e^{in\theta}\rbrace$ in $L^2(\mathbb S^1)$,
\begin{align*}
&\avgint_{\mathbb S^1}\Re{(\partial_\eta\bar v)^2}\, d\theta \\
&= \Re{ 
\sum_{n,m\in\mathbb Z}
\left(\bar w_n' +\frac{1+n}{r}\bar w_n\right)\left(\bar w_m' +\frac{1+m}{r}\bar w_m\right) \avgint_{\mathbb S^1} e^{-i(n+m)\theta} d\theta
} \\
& =\Re{
\sum_{n\in\mathbb Z}
\left(\bar w_n' +\frac{1+n}{r}\bar w_n\right)\left(\bar w_{-n}' +\frac{1-n}{r}\bar w_{-n}\right)
} \\
&=\sum_{n\in\mathbb Z} \Re{
 \left(\bar w_n' +\frac{1+n}{r}\bar w_n\right)\left(\bar w_{-n}' +\frac{1-n}{r}\bar w_{-n}\right) 
}.
\end{align*}
This implies the announced orthogonality and completes the proof of Lemma~\ref{l:decomp}.
\end{proof}

According to the decomposition of Lemma~\eqref{l:decomp}, we define the quadratic forms
\begin{align*}
Q_0^\delta[\varphi]&=\frac 1{2\pi} \mathcal Q^\delta\left[\varphi(r)e^{i\theta}\right] & \text{for }\varphi\in \H_0,\\
Q_n^\delta[\varphi,\psi]&=\frac 1{2\pi}\mathcal Q^\delta\left[e^{i\theta}\left(\varphi(r)e^{in\theta} +\psi(r)e^{-in\theta} \right)\right] &\text{for }(\varphi,\psi)\in \H_1,
\end{align*}
where $\H_0$ and $\H_1$ are the natural spaces corresponding to the conditions $\varphi(r)e^{i\theta}\in\H$ and $e^{i\theta}\left(\varphi(r)e^{in\theta} +\psi(r)e^{-in\theta}\right) \in\H$ for $n\geq 1$, respectively.
\begin{align*}
	\H_0
	&=\left\lbrace \varphi\in H^1_{loc}(0,\infty)\colon \int_0^{+\infty}\left(\abs{\varphi'}^2+\frac{\abs{\varphi}^2}{r^2} +\Re{\varphi}^2\right) r\, dr <+\infty\right\rbrace,\\
	\H_1
	&=\left\lbrace (\varphi,\psi) \in (H^1_{loc}(0,\infty))^2\colon \int_0^{+\infty}\left(\abs{\varphi'}^2+\abs{\psi'}^2+\frac{\abs{\varphi}^2+\abs{\psi}^2}{r^2} +\abs{\varphi+\bar\psi}^2\right) r\, dr <+\infty\right\rbrace
\end{align*}
\begin{remark}\label{r:densitytest}
Using the density of smooth functions in $H^1_{loc}$ and cut-off functions $\chi_\e$ such that $\mathbf 1_{2\e<r<\e^{-1}}\leq \chi_\e(r) \leq \mathbf 1_{\e<r<2\e^{-1}}$ and $|\chi_\e'(r)|\leq C/r$, we see that smooth test functions with compact support in $(0,\infty)$ are dense in $\H_0$ and $\H_1$. Hence, in the sequel, we will always be able to perform calculations assuming, without loss of generality, that $\varphi$ and $\psi$ are such test functions.
\end{remark}

The quadratic forms $Q_0^\delta$ and $Q_n^\delta$ are explicitly given by
\begin{align}
\label{eq:Q0delta}
Q_0^\delta[\varphi]
&=
\int_0^\infty \Bigg[
\abs{\varphi'}^2+\frac{1}{r^2}\abs{\varphi}^2 
+ \delta\Re{ \left(\bar\varphi'+\frac 1r\bar\varphi\right)^2} \\
&\hspace{7em}
+(1+\delta)\left\lbrace 2 f_0^2 (\Re{\varphi})^2 - (1-f_0^2)\abs{\varphi}^2\right\rbrace
\Bigg]\, rdr ,\nonumber\\
\label{eq:Qndelta}
Q_n^\delta[\varphi,\psi ]
&=\int_0^\infty \Bigg[
\abs{\varphi'}^2+\abs{\psi'}^2+\frac{(1+n)^2}{r^2}\abs{\varphi}^2 +\frac{(1-n)^2}{r^2}\abs{\psi}^2 \\
&\hspace{5em}
+ 2\delta\Re{ \left(\bar\varphi'+\frac {1+n}r\bar\varphi\right)\left(\bar\psi'+\frac {1-n}r\bar\psi\right)} \nonumber\\
&\hspace{7em}
+(1+\delta)\left\lbrace f_0^2 \abs{\varphi +\bar\psi}^2 - (1-f_0^2)\left(\abs{\varphi}^2+\abs{\psi}^2\right)\right\rbrace
\Bigg]\, rdr .\nonumber
\end{align}

\begin{remark}\label{r:splitrealim}
For every $n\geq 1$ there is a further splitting, namely
\begin{align*}
Q_n^\delta[\varphi,\psi]&=Q_n^\delta\left[\Re{\varphi},\Re{\psi}\right] + Q_n^\delta\left[\Im{\varphi},-\Im{\psi}\right].
\end{align*}
Consequently, it will be sufficient to consider real-valued test functions $\varphi,\psi$.
\end{remark}

\section{Study of  the lower modes $Q_0^\delta$ and $Q_1^\delta$}\label{s:low}

We show that $Q_0^\delta$ is positive for $\delta\leq 0$, but it can become negative for $\delta>0$. In addition, we prove that $Q_1^\delta$ is nonnegative for all $\delta\in (-1,0]$.

\subsection{Positivity of $Q_0^\delta$ for $\delta\in (-1,0]$}

Let us recall from~\eqref{eq:Q0delta} that $Q_0^\delta$ is given by
\begin{align*}
Q_0^\delta[\varphi]
&=
\int_0^\infty \Bigg[
\abs{\varphi'}^2+\frac{1}{r^2}\abs{\varphi}^2 
+ \delta\Re{ \left(\bar\varphi'+\frac 1r\bar\varphi\right)^2} \\
&\hspace{7em}
+(1+\delta)\left\lbrace 2 f_0^2 (\Re{\varphi})^2 - (1-f_0^2)\abs{\varphi}^2\right\rbrace
\Bigg]\, rdr
\end{align*}
We now introduce the  quadratic form
\begin{align*}
A_0[\varphi]&:=Q_0^0[\varphi]\\
&= \int_0^\infty \Bigg[
\abs{\varphi'}^2+\frac{1}{r^2}\abs{\varphi}^2 
+2 f_0^2 (\Re{\varphi})^2 - (1-f_0^2)\abs{\varphi}^2
\Bigg]\, rdr.
\end{align*}
It is known that $A_0[\varphi]>0$, unless $\varphi=0$ (see Appendix~\ref{a:stabA} for more details). 
Moreover, we have the identity
\begin{align*}
Q_0^\delta[\varphi]&=(1+\delta) A_0[\Re{\varphi}]+(1-\delta)A_0[i\Im{\varphi}]
-2\delta\int (1-f_0^2) (\Im{\varphi})^2\, r\, dr
\\
&\quad
+\delta \int_0^\infty \frac{d}{dr}\left[(\Re{\varphi})^2-(\Im{\varphi})^2\right]\, dr\\
&=(1+\delta) A_0[\Re{\varphi}]+(1-\delta)A_0[i\Im{\varphi}]
-2\delta\int (1-f_0^2) (\Im{\varphi})^2\, r\, dr,
\end{align*}
which is valid for any $\varphi\in C_c^\infty(0,\infty)$, hence for $\varphi\in\H_0$ thanks to Remark~\ref{r:densitytest}.
Since $1-f_0^2 \geq 0$, we deduce the positivity of $Q_0^\delta$ for every $\delta\in (-1,0]$.

\subsection{Instability for $\delta>0$}

Using the formula \eqref{eq:A0dec} obtained for $A_0$ in Appendix~\ref{a:stabA}, we see that for any compactly supported real-valued test function $\chi$ we have
\begin{align*}
Q_0^\delta[if_0\chi]&=(1-\delta)\int f_0^2 (\chi')^2 \, r\, dr -2\delta \int (1-f_0^2)f_0^2\chi^2\, r\, dr.
\end{align*}
Applying this to $\chi_n(r)=\chi_1(r/n)$, for a fixed test function $\chi_1$, and using the asymptotic expansion  
\cite{herve94,chen94}:
\begin{align*}
f_0(r)=1-\frac{1}{2r^2} +O(r^{-4})\qquad\text{ as }r\to\infty,
\end{align*}
we see that
\begin{align*}
\lim_{n\to\infty} Q_0^\delta[if_0\chi_n]=(1-\delta)\int (\chi_1')^2 \, r\, dr - 2\delta \int \frac{\chi_1^2}{r^2}\, r\, dr.
\end{align*}
When $\delta>0$, this expression must be negative for some $\chi_1$, since Hardy's inequality is known to fail in two dimensions. Explicitly, by choosing
\begin{align*}
\chi_1(r)=\sin(\sqrt\lambda \ln r)\mathbf 1_{(1,e^{\pi/\sqrt\lambda})}(r)\qquad\text{for }\lambda=\frac{\delta}{1-\delta}>0,
\end{align*}
we have that $\chi_1\in H^1(0,\infty)$ is compactly supported, and
\begin{align*}
\lim_{n\to\infty} Q_0^\delta[if_0\chi_n]=-\delta \int \frac{\chi_1^2}{r^2}\, r\, dr <0.
\end{align*}
Whence, for $\delta>0$, the mode of order 0 already brings instability. This comes as no surprise as this mode corresponds to infinitesimal rotations (see Appendix~\ref{a:stabA}), and we know that the far-field behavior $e^{i\theta}$ is unstable: rotating this far-field behavior decreases the energy.

\subsection{Positivity of $Q_1^\delta$ for $\delta\leq 0$}\label{ss:Q1deltaneg}

Recall, according to~\eqref{eq:Qndelta}, that $Q_1^\delta$ is given by
\begin{align*}
Q_1^\delta[\varphi,\psi ]
&=\int_0^\infty \Bigg[
\abs{\varphi'}^2+\abs{\psi'}^2+\frac{4}{r^2}\abs{\varphi}^2  \\
&\hspace{5em}
+ 2\delta\Re{ \left(\bar\varphi'+\frac {2}r\bar\varphi\right) \bar\psi' } \\
&\hspace{7em}
+(1+\delta)\left\lbrace f_0^2 \abs{\varphi +\bar\psi}^2 - (1-f_0^2)\left(\abs{\varphi}^2+\abs{\psi}^2\right)\right\rbrace
\Bigg]\, rdr .
\end{align*}
We introduce the quadratic form $A_1:=Q_1^0$, namely
\begin{align*}
A_1[\varphi,\psi]& =
\int_0^\infty \Bigg[
\abs{\varphi'}^2+\abs{\psi'}^2+\frac{4}{r^2}\abs{\varphi}^2\\
&\hspace{5em}
+ f_0^2 \abs{\varphi +\bar\psi}^2 - (1-f_0^2)\left(\abs{\varphi}^2+\abs{\psi}^2\right)
\Bigg]\, rdr.
\end{align*}
It is a known fact that $A_1$ is nonnegative on $\H_1$, and vanishes exactly at pairs $(\varphi,\psi)$ corresponding to maps $v$ which are linear combinations of $\partial_x\uiso$ and $\partial_y\uiso$ (see Appendix~\ref{a:stabA} for more details). Moreover, we have
\begin{align}\label{eq:Q1deltaA1}
Q_1^\delta[\varphi,\psi]-(1+\delta)A_1[\varphi,\psi]&=
-\delta\int_0^\infty \Bigg[
\abs{\varphi'}^2+\abs{\psi'}^2+\frac{4}{r^2}\abs{\varphi}^2  \Bigg]\, rdr \\
&\quad + 2\delta \int_0^\infty\Re{ \left(\bar\varphi'+\frac {2}r\bar\varphi\right) \bar\psi' } \, rdr\nonumber
\\
&
=-\delta\int_0^\infty
\abs{\varphi'+\frac 2r \varphi -\bar\psi'}^2\, rdr -2\delta\int_0^\infty \frac{d}{dr}\left[|\varphi|^2\right]\, dr \nonumber \\
& = -\delta\int_0^\infty
\abs{\varphi'+\frac 2r \varphi -\bar\psi'}^2\, rdr,\nonumber
\end{align}
for $(\varphi,\psi)\in (C_c^\infty(0,\infty))^2$, hence for all $(\varphi,\psi)\in\H_1$.
From this identity we infer that  $Q_1^{\delta}\geq 0$ for every $\delta\in(-1, 0]$, and equality can only occur when $v$ is a linear combination of $\partial_x\uiso$ and $\partial_y\uiso$.

\section{Study of the higher modes $Q_n^\delta$ for $n\geq 2$}\label{s:high}

\subsection{Positivity of $Q_n^\delta$ for $n\geq 2$ and $\delta\in [-1/\sqrt 5,0]$}

Let us recall: in the isotropic case, the positivity of $Q_n^\delta$ (any $n\geq 2$) is a consequence of the fact that $Q_n^0\geq Q_1^0$. Here, from the definition~\eqref{eq:Qndelta} of $Q_n^\delta$,  we have
\begin{align}\label{eq:QnQ1}
&Q_n^\delta[\varphi,\psi]-Q_1^\delta[\varphi,\psi]\\
&=(n-1)\int_0^\infty 
\Bigg[
\frac{n+3}{r^2}\abs{\varphi}^2 +\frac{n-1}{r^2}\abs{\psi}^2 
-2\delta \frac{n+1}{r^2} \Re{\bar\varphi\bar\psi} \nonumber\\
&\hspace{17em} +2\frac{\delta}{r} \Re{\bar\varphi\bar\psi'-\bar\varphi'\bar\psi}
\Bigg]\, rdr.\nonumber
\end{align}
Unlike what happens in the isotropic case, this does not obviously have a sign (because of the last term which contains derivatives).

It seems reasonable to use a decomposition for $\varphi,\psi$ adapted to $Q_1^\delta$, as in Appendix~A. Accordingly, we define for any real-valued test functions $\zeta,\eta$, the adapted quadratic form
\begin{align*}
B_n^\delta[\zeta,\eta]&=\frac 12 Q_n^\delta\left[f_0'\zeta - r^{-1}f_0\eta,f_0'\zeta +r^{-1}f_0\eta \right]
\end{align*} 
Decomposing 
\[
	Q_n^\delta=(1+\delta)A_1 +Q_1^\delta -(1+\delta)A_1 + Q_n^\delta-Q_1^\delta
\]
and using the above expressions of $Q_n^\delta-Q_1^\delta$ \eqref{eq:QnQ1} and $Q_1^\delta-(1+\delta)A_1$ \eqref{eq:Q1deltaA1}, we have, for real-valued $(\varphi,\psi)\in \H_1$:
\begin{align*}
Q_n^\delta[\varphi,\psi]&=(1+\delta)A_1[\varphi,\psi] \\
&\quad
-\delta
\int_0^\infty
\left(\varphi'+\frac 2r \varphi -\psi'\right)^2\, rdr  \\
&\quad  +(n-1)\int_0^\infty 
\left[
\frac{n+3}{r^2} \varphi^2 +\frac{n-1}{r^2} \psi^2 
-2\delta \frac{n+1}{r^2}  \varphi \psi \right] \\
&\quad  +2\delta (n-1) \int_0^\infty \frac{1}{r} \left(\varphi \psi'-\varphi'\psi\right)
\, rdr.
\end{align*}
When plugging in $\varphi= f_0'\zeta - r^{-1}f_0\eta$, $\psi = f_0'\zeta +r^{-1}f_0\eta$, the first term significantly simplifies thanks to the formula \eqref{eq:A1dec} for $A_1$  in Appendix~\ref{a:stabA}. For the other terms we directly expand 
\begin{align*}
&\varphi'+\frac 2r \varphi -\psi'
=2f_0'\frac{\zeta-\eta}{r} -2 \frac{f_0}{r}\eta', \\
&\frac{n+3}{r^2} \varphi^2 +\frac{n-1}{r^2} \psi^2 
-2\delta \frac{n+1}{r^2}  \varphi \psi \\
&=2(1-\delta)\frac{n+1}{r^2}(f_0'\zeta)^2 +2(1+\delta)\frac{n+1}{r^2}\left(\frac{f_0}{r}\eta\right)^2
-\frac{8}{r^2}f_0'\zeta\frac{f_0}{r}\eta  \\
&\varphi \psi'-\varphi'\psi  =2 \left(\frac{f_0}{r}\eta\right)'f_0'\zeta -2 (f_0'\zeta)'\frac{f_0 }{r}\eta,
\end{align*}
from which it follows that
$B_n^\delta[\zeta,\eta]=(1/2) Q_n^\delta [f_0'\zeta - r^{-1}f_0\eta,f_0'\zeta +r^{-1}f_0\eta  ]$ can be rewritten as
\begin{align}\label{eq:Bndelta}
B_n^\delta[\zeta,\eta] 
&=(1+\delta)\int_0^\infty \left[
\frac{f_0^2}{r^2}(\eta')^2 +(f_0')^2(\zeta')^2+\frac 2 {r^3} f_0f_0'(\eta - \zeta)^2
\right]\, rdr \\
&\quad -2\delta\int_0^\infty \left[\frac{f_0'}{r}\left(\eta-\zeta\right) +\frac{f_0}{r}\eta'\right]^2
\, rdr \nonumber\\
&\quad + (n-1)\int_0^\infty\left[
(1-\delta)\frac{n+1}{r^2}(f_0'\zeta)^2 +(1+\delta)\frac{n+1}{r^2}\left(\frac {f_0}r\eta\right)^2
-\frac{4}{r^2}\left(f_0'\zeta\right)\left(\frac {f_0}r \eta\right)
\right]\, rdr \nonumber\\
&\quad + 2\delta(n-1)\int_0^\infty \frac 1r \left[\left(\frac {f_0}r \eta\right)' f_0'\zeta-\left(f_0'\zeta\right)'\frac {f_0}r \eta  \right]\, rdr. \nonumber
\end{align}
%
%
Integrating by parts, the last integral becomes 
\begin{align*}
\int_0^\infty \frac 1r \left[\left(\frac {f_0}r \eta\right)'f_0'\zeta-\left(f_0'\zeta\right)'\frac {f_0}r \eta  \right]\, rdr & = 2\int_0^\infty  \left(\frac{f_0}{r}\eta\right)'f_0'\frac\zeta r \, rdr \\
&=2\int_0^\infty \left[\left(f_0'-\frac{f_0}{r}\right)f_0'\frac\eta r\frac\zeta r +\frac{f_0}{r}\eta'f_0'\frac\zeta r \right]\, rdr.
\end{align*}
We use the  first positive term in \eqref{eq:Bndelta} in order to absorb this latter term: thanks to the identity
\begin{align*}
(1+\delta)\frac{f_0^2}{r^2}(\eta')^2+4\delta(n-1)\frac{f_0}{r}\eta'f_0'\frac\zeta r 
& =(1+\delta)\left(\frac{f_0}{r}\eta' +\frac{2\delta}{1+\delta}(n-1)f_0'\frac{\zeta}{r}\right)^2 \\
&\quad -4\frac{\delta^2}{1+\delta}(n-1)^2 \left(f_0'\right)^2\left(\frac\zeta r\right)^2,
\end{align*}
we rewrite \eqref{eq:Bndelta} as
\begin{align*}
B_n^\delta[\zeta,\eta]
&=B^{\delta,1}_n[\zeta,\eta] + (n-1)B^{\delta,2}_n[\zeta,\eta],\\
B^{\delta,1}_n[\zeta,\eta]&=(1+\delta)\int_0^\infty \left[
\left(\frac{f_0}{r}\eta' +\frac{2\delta}{1+\delta}(n-1)f_0'\frac{\zeta}{r}\right)^2
 +
(f_0')^2(\zeta')^2\right]
\, rdr \\
&\quad +2\int_0^\infty  \left\lbrace(1+\delta) f_0'\frac{f_0}{r}\frac{(\eta-\zeta)^2}{r^2}-\delta \left[\frac{f_0'}{r}\left(\eta-\zeta\right) +\frac{f_0}{r}\eta'\right]^2
\right\rbrace\, rdr  ,\\
B^{\delta,2}_n[\zeta,\eta]&=\int_0^\infty q^\delta_n(r)\left[f_0'\frac\zeta r,\frac{f_0}{r}\frac\eta r\right] \, rdr,
\end{align*}
and $q^\delta_n(r)$ is the  quadratic form on $\R^2$ given by 
\begin{align*}
q^\delta_n(r)[X,Y]&=a_n  X^2 +b_n Y^2 +2 c(r) XY,\\
a_n&= (1-\delta)(n+1)-4\frac{\delta^2}{1+\delta}(n-1)\\
b_n &
= (1+\delta)(n+1) \\
c(r)&=-2 -2\delta \left(1- r\frac{f_0'}{f_0}\right)
\end{align*}
We readily see that $B^{\delta,1}_n$ is nonnegative for $\delta\leq 0$. Moreover, since $1>rf_0'/f_0>0$  \cite[Proposition~2.2]{ignat14ode}, for $\delta\leq 0$, it follows that
\begin{align*}
\abs{c(r)}\leq 2.
\end{align*}
As $b_n>0$, a sufficient condition for $q_n^\delta(r)$ to be positive definite for all $r>0$ is
\begin{align*}
 4< a_nb_n =(1-\delta^2)(n+1)^2 -4\delta^2(n^2-1).
\end{align*}
This amounts to the condition
\[
0< \alpha(\delta)n^2+ \beta(\delta) n +\gamma(\delta),
\]
where
\begin{align*}
\alpha(\delta)&=1-5\delta^2,\\
 \beta(\delta)&=2(1-\delta^2),\\
\gamma(\delta)&=-3(1-\delta^2).
\end{align*}
For $\delta\in [-1/\sqrt{5},0]$ we have $\alpha(\delta),\beta(\delta)\geq 0$ so that the above polynomial in $n$ is nondecreasing on $[0,+\infty)$. Hence, it is positive for all values of $n\geq 2$ if and only if it is positive for $n=2$. That is,
\begin{align*}
0&< 4\alpha(\delta) +2\beta(\delta)+\gamma(\delta)=5- 21 \,\delta^2.
\end{align*}
We deduce that  $q_n^\delta$ is a positive definite quadratic form for all $n\geq 2$ whenever $\delta\in [-1/\sqrt{5},0]$.
 In particular, $B_n^{\delta,2} \geq 0$ and therefore $Q_n^\delta \geq 0$ for $\delta\in [-1/\sqrt{5},0]$, with equality  only at $(0,0)$.

\subsection{Instability for $\delta\approx - 1$}

In this section we show that $Q_n^\delta$ can take negative values for $\delta\approx -1$ and $n\geq 1$ large enough. To this end, we choose $\eta=\zeta$ in \eqref{eq:Bndelta}, to obtain
\begin{align*}
\hat B^\delta_n[\zeta]&= B^\delta_n[\zeta,\zeta] \\
&=(1-\delta)\int_0^\infty \frac{f_0^2}{r^2}(\zeta')^2\, rdr +(1+\delta)\int_0^\infty (f_0')^2(\zeta')^2 \, rdr + (n-1)\int_0^\infty\frac{\zeta^2}{r^2} \alpha_n^\delta(r) \, rdr\\
\alpha^\delta_n(r)&=
(1-\delta)(n+1)(f_0')^2 +(1+\delta)(n+1)\left(\frac {f_0}r\right)^2 
-2(2+\delta)f_0'\frac {f_0}r +2\delta (f_0')^2 -2\delta f_0f_0''.
\end{align*}
Using the asymptotics of $f_0$ (\!\!\cite{chen94,herve94})
\begin{align*}
f_0(r)=1-\frac 12 r^{-2}  +O(r^{-4}),\quad
f_0'(r)=r^{-3} +  O(r^{-5}),\quad
f_0''(r)=-3 r^{-4} + O(r^{-6}),
\end{align*}
 we find, for $r\to+\infty$,
\begin{align*}
\alpha^\delta_n(r)= \frac{(1+\delta)(n+1)}{r^2}\left(1-\frac{1}{r^2}\right) -4\frac{1-\delta }{r^4}  +O(r^{-6}).
\end{align*}
For $\delta=-1$ the leading order is negative.
Hence, there exists $\e>0$ and a compact interval $[r_0,r_0+1]$  on which $\alpha_n^{-1}\leq -2\varepsilon$. Thus, we deduce that for all $n\geq 2$ there exists $\delta_n >-1$ such that for all $\delta\in (-1,\delta_n]$,
\begin{align*}
-\e &\geq \alpha^\delta_n(r),\qquad\forall r\in [r_0,r_0+1].
\end{align*}
Choosing a nonzero test function $\zeta_0$ with support in $[r_0,r_0+1]$, we obtain
\begin{align*}
\hat B^\delta_n[\zeta_0]&\leq C_1(\zeta_0)-(n-1)\e C_2(\zeta_0)\qquad\forall \delta\in (-1,\delta_n],
\end{align*}
for some $C_1(\zeta_0),C_2(\zeta_0)>0$.
If $n$ is large enough this becomes negative. Compared to the isotropic case this is a really new situation: lower modes are positive but higher modes can bring instability. 

\section{Proof of Theorem~\ref{t:main}}\label{s:proof}

In what precedes we have shown that $\u$ is nondegenerately stable for small $\delta\leq 0$, and unstable for $\delta>0$ and $\delta$ close to $- 1$. In particular, setting
\begin{align*}
\delta_1=\sup \lbrace \delta\in (-1,0)\colon \u \text{ is unstable }\rbrace,
\end{align*}
we know that $-1<\delta_1<0$. It remains to show that $\u$ is unstable for all $\delta\in (-1,\delta_1)$, and nondegenerately stable for $\delta\in (\delta_1,0]$.

Let $\delta'\in (-1,\delta_1)$ be such that $\u$ is unstable, that is, $\mathcal Q^{\delta'}[v]<0$ for some choice of $v\in H$. Given that $\delta\mapsto \mathcal Q^\delta[v]$ is an affine function which is nonnegative for $\delta=0$ and negative for $\delta=\delta'$, we deduce that $\mathcal Q^\delta[v]<0$ for all $\delta\leq \delta'$. Therefore, $\u$ is unstable for all $\delta\in (-1,\delta')$. By arbitrariness of $\delta'$ we deduce that $\u$ is unstable for all $\delta\in (-1,\delta_1)$. 

Let us now fix $\delta\in (\delta_1,0]$.
By definition of $\delta_1$, $\u$ is not unstable for all $\delta\in (\delta_1,0]$. In other words, $\mathcal Q^\delta[v]$ is nonnegative for all $v\in \H$. It remains to show that, in fact, $\mathcal Q^\delta[v]>0$ for all $v\in \H\setminus \mathrm{span}(\partial_x\uiso,\partial_y \uiso)$. 
We observe that the function $\delta\mapsto \mathcal Q^\delta[v]$ is affine for any given $v\in \H\setminus \mathrm{span}(\partial_x \uiso,\partial_y\uiso)$; it is positive for $\delta=0$ because $\uiso$ is nondegenerately stable, and it is nonnegative for $\delta\in (\delta_1,0)$. Thus, it must be strictly positive for $\delta\in (\delta_1,0)$. 
This proves the desired nondegenerate stability in the announced range.


\appendix

\section{Positivity of $A_0,A_1$}\label{a:stabA}

We sketch here the approach in \cite{mironescu95stab}, adapted to our notation (see also \cite{delpino04}), based on Hardy-type decompositions to show positivity of the two following quadratic forms
\begin{align*}
A_0[\varphi]&= \int_0^\infty \Bigg[
\abs{\varphi'}^2+\frac{1}{r^2}\abs{\varphi}^2 \\
&\hspace{5em}
+2 f_0^2 (\Re{\varphi})^2 - (1-f_0^2)\abs{\varphi}^2
\Bigg]\, rdr,\nonumber\\
A_1[\varphi,\psi]& = 
\int_0^\infty \Bigg[
\abs{\varphi'}^2+\abs{\psi'}^2+\frac{4}{r^2}\abs{\varphi}^2
\\
&\hspace{5em}
+ f_0^2 \abs{\varphi +\bar\psi}^2 - (1-f_0^2)\left(\abs{\varphi}^2+\abs{\psi}^2\right)
\Bigg]\, rdr.\nonumber
\end{align*}

Testing equation \eqref{eq:f0}, solved by $f_0$, against $f_0\abs{\tilde\varphi}^2$ for any smooth compactly supported $\tilde\varphi\in C^{\infty}_{c}(\R;\C)$,  one obtains
\begin{align*}
\int_0^\infty \left[ 
(f_0')^2|\tilde\varphi|^2+2f_0f_0'\tilde\varphi\cdot\tilde\varphi' +\frac{f_0^2}{r^2}|\tilde\varphi|^2-(1-f_0^2)f_0^2 |\tilde\varphi|^2
\right]\, r dr=0,
\end{align*}
so that
\begin{align}\label{eq:A0dec}
A_0[f_0\tilde\varphi]&=\int_0^\infty \left[f_0^2\abs{\tilde\varphi'}^2 +2f_0^4(\Re{\tilde \varphi})^2\right] \,rdr.
\end{align}
By density of test functions, and since $f_0>0$, we deduce that $A_0[\varphi]>0$ for any non-zero $\varphi\in \H_0$. Moreover $A_0[\varphi]\approx 0$ exactly when $\varphi\approx if_0$. This corresponds to the fact that in the isotropic case $\delta=0$, 
\begin{align*}
\partial_\alpha[e^{i\alpha}\u]_{\lfloor\alpha=0}=if_0e^{i\theta}
\end{align*}
solves the linearized equation due to rotational invariance.
 
For $A_1$, it is convenient to start by splitting it as
\begin{align*}
A_1[\varphi,\psi]&=A_1\left[\Re{\varphi},\Re{\psi}\right]+ A_1\left[\Im{\varphi},
-\Im{\psi}\right],
\end{align*}
so we may just treat the case of real-valued test functions $\varphi,\psi$. 
Guided by the fact that
\begin{align*}
\partial_x \uiso = e^{i\theta}( f_0' \cos\theta - i\frac {f_0}r \sin\theta )
,\qquad
 \partial_y \uiso =e^{i\theta} (f_0'\sin\theta +i\frac {f_0}r\cos\theta),
 \end{align*}
solve the linearized equation around $\uiso$, one  uses the ansatz
\begin{align*}
\varphi=f_0'\zeta - \frac{f_0}{r}\eta,\qquad\psi=f_0'\zeta + \frac{f_0}{r}\eta,
\end{align*}
for some real-valued $\eta,\zeta\in C_c^\infty(0,\infty)$.
Testing equation \eqref{eq:f0}, solved by $f_0$, against $f_0r^{-2}\eta^2$  we obtain
\begin{align*}
\int_0^\infty\left[
\left(\left(\frac{f_0}{r}\right)'\right)^2\eta^2 
+ 2 \left(\frac{f_0}{r}\right)'\frac{f_0}{r}\eta\eta' 
+\frac{2}{r^4}f_0^2\eta^2
-\frac{2}{r^3}f_0f_0'\eta^2 
-(1-f_0^2)\frac{f_0^2}{r^2}\eta^2
\right]\,rdr =0,
\end{align*}
and similarly testing \eqref{eq:f0} against $(f_0'\zeta^2)'$ we find
\begin{align*}
\int_0^\infty\left[
(f_0'')^2\zeta^2
+2 f_0'f_0''\zeta\zeta'
+ \frac{2}{r^2}(f_0')^2\zeta^2 
-\frac{2}{r^3}f_0f_0'\zeta^2
+(3f_0^2-1)(f_0')^2\zeta^2
\right]\, rdr =0.
\end{align*}
As a consequence of these two identities, we learn
\begin{align}\label{eq:A1dec}
& A_1 \left[f_0'\zeta - r^{-1}f_0\eta,f_0'\zeta + r^{-1}f_0\eta\right] \\
&=2\int_0^\infty \left[
\frac{f_0^2}{r^2}(\eta')^2 +(f_0')^2(\zeta')^2+\frac 2 {r^3} f_0f_0'(\eta - \zeta)^2
\right]\, rdr.\nonumber
\end{align}
Since $f_0,f_0'>0$ one may consider the choice
\begin{align*}
&\zeta=\frac{1}{2f_0'}(\varphi+\psi),\quad\eta=\frac{r}{2f_0}(\psi-\varphi),
\end{align*}
and deduce from the above that $A_1[\varphi,\psi]>0$ for all non-zero  $(\varphi,\psi)\in \H_1$. Moreover $A_1[\varphi,\psi]= 0$ exactly when $(\varphi,\psi)$ is in the real linear span of
\begin{align*}
\left(f_0'-\frac {f_0}r,f_0'+\frac {f_0}r\right),\quad 
\left(i\left(f_0'-\frac {f_0}r\right),-i\left(f_0'+\frac {f_0}r\right)\right),
\end{align*} 
which corresponds to the fact that $\partial_x\uiso$ and $\partial_y\uiso$
solve the linearized equation.
\bigskip
\bigskip

\bibliographystyle{acm}
\bibliography{ref}

\end{document}